\documentclass[12pt]{article}

\usepackage{amsmath, amsfonts, amssymb}
\usepackage[T2A]{fontenc}
\usepackage[utf8]{inputenc}
\usepackage[english]{babel}
\usepackage[title]{appendix}
\usepackage{todonotes}
\usepackage{comment}
\usepackage{enumitem}
\usepackage[left=3cm, right=2cm, top=2cm, bottom=2cm, bindingoffset=0cm, marginparwidth=3cm]{geometry}

\usepackage{hyperref}
\usepackage{float}
\usepackage{amsthm}

\DeclareMathOperator{\id}{id}

\newcounter{theorems}

\newtheorem{thm}[theorems]{Theorem}
\newtheorem{lem}[theorems]{Lemma}

\theoremstyle{definition}
\newtheorem{defin}[theorems]{Definition}
\theoremstyle{remark}

\newtheorem{rem}[theorems]{Remark}
\newtheorem*{rem*}{Remark}

\makeatletter
\def\blfootnote{\gdef\@thefnmark{}\@footnotetext}
\makeatother

\def\e {\varepsilon }

\def\phi {\varphi}
\def\be {\begin{equation}}
\def\ee {\end{equation}}
\def\bt {\begin{thm}}
\def\et {\end{thm}}
\def\R {\mathcal R}

\def\U{\mathcal U}

\def\t{{\theta}}

\def\Int{\mathrm{Int}}
\def\ef {f}

\DeclareMathOperator{\intt}{Int}

\begin{document}

\title{Attractors with non-invariant interior}
\author{Stanislav Minkov\footnote{Brook Institute of Electronic Control Machines, Moscow, Russia}, \; Alexey Okunev\footnote{Pennsylvania State University, State College, PA, USA}, \;Ivan Shilin\footnote{HSE University, Moscow, Russia}\hphantom{1}\thanks{Corresponding author; e-mail: i.s.shilin@yandex.ru}}
\date{}

\maketitle

\begin{abstract}
We construct an open set of endomorphisms of an arbitrary two-dimensional manifold which have attractors and non-wandering sets with non-invariant interior. 
This is a notable contrast to the properties of diffeomorphisms, where the interior must be invariant.
\end{abstract}



\blfootnote{{\bf Funding:} SM and IS were partially supported by the RFBR grant 20-01-00420-a.}

\section{Introduction}
The properties of generic endomorphisms are somewhat different from the properties of generic diffeomorphisms. It is conjectured that $C^1$-generic diffeomorphisms (of a connected manifold) whose non-wandering set has a non-empty interior are transitive,~\cite{ABD}. In contrast, for endomorphisms there are known open examples of attractors with non-empty interior (see e.g. Volk\cite{V} and Tsujii~\cite{T}).

We build on the ideas of these examples to show that the interior of the nonwandering set or attractor can be not only non-empty, but also non-invariant, and in a persistent way. That is, we construct an open set of maps that take a point in the interior of the attractor to the boundary of it. 
This is another contrast with diffeomorphisms, as for a diffeomorphism the interior of an invariant compact set is always invariant. In the known examples of attractors with non-empty interior the interior also is invariant.

Informally, an attractor is a subset of the phase space that attracts many points. There are different ways to give a formal definition. In this article we deal with the following two definitions proposed by John Milnor in~\cite{M}.

Let $X$ be a smooth compact manifold and $f: X \mapsto X$ be a $C^1$-map. We assume that $X$ is endowed with a Borel probability measure equivalent to the Lebesgue measure in each coordinate chart, and refer to it as the Lebesgue measure on~$X$. Note that the Milnor attractor we are about to define does not depend on the choice of this measure.    

\begin{defin}[Milnor attractor\footnote{J. Milnor originally called it `the likely limit set'.}] 
The Milnor attractor~$A_{Mil}(f)$ of the dynamical system $(X, f)$ is the smallest (by inclusion) closed set that contains the $\omega$-limit sets of almost all points with respect to the Lebesgue measure on $X$.
\end{defin}

\begin{defin}[Generic limit set] 
The generic limit set~$A_{gen}(f)$ of the dynamical system $(X, f)$ is the smallest closed set that contains the $\omega$-limit set of a generic point, i.e., there is a residual subset $R \subset X$ with $\omega(x) \subset A_{gen}$ for all $x \in R$.
\end{defin}

\begin{thm}[Main theorem]\label{thm:main}
For any two-dimensional surface $M$, there is a $C^1$-open set $\mathcal U$ of $C^1$-endomorphisms of $M$ with the following properties.
\begin{itemize}
    \item For every endomorphism in $\mathcal U$ the non-wandering set and the generic limit set both have non-invariant interior.
    \item There is a residual subset $\mathcal R \subset \mathcal U$ formed by maps for which the Milnor attractor has non-invariant interior. The set $\mathcal R$ includes the set  $\mathcal{U}_2$ of all $C^2$-maps in $\mathcal U$. Thus, there is a $C^2$-open set $\mathcal{U}_2$ of maps that have Milnor attractors with non-invariant interior.
\end{itemize}
\end{thm}

\begin{rem}
\begin{enumerate}
  \item We want our phase space to be a connected manifold, and this is why we need a non-trivial construction. In Appendix~\ref{s:non-manifold} we present a simple example of a map on a union of a point and a segment with the same property.
  
  \item In our example the image of the interior of the attractor is inside the closure of this interior. This holds for any map $f$ such that non-singular points (i.e., points at which the Jacobian of $f$ is non-zero) are dense in the phase space. Indeed, for any point $x$ of the interior one can consider a sequence of non-singular points $x_n$ converging to $x$; all points $x_n$ are mapped into the interior of the attractor, so $x$ must be mapped in the closure of the interior. Note that non-singular points are dense in the phase space for a $C^r$-residual set of maps ($r=1, 2, \dots$). 

  \item Generically, the attractors of endomorphisms of one-dimensional connected manifolds either have empty interior or coincide with the whole phase space, as we show in Appendix~\ref{s:1D} using the results of~\cite{J}. Therefore, our example has the smallest possible dimension.
  
  \item We focus on the smallest dimension where a robust example of this type is possible. However, the main theorem is also true for any manifold of dimension higher than 2. We prove this in Appendix~\ref{s:3d}.

  \item Our initial approach was to first construct a skew product over a circle extension with the required properties, and then perturb it in the class of endomorphisms and use the technique from~\cite{IN} to regain the structure of a skew product and prove that the non-invariance properties are persistent. Later we came up with a simpler geometric argument. Nevertheless, the ``initial'' skew product will appear in Section~\ref{sec:skew}.
\end{enumerate}
\end{rem}

\section{Proof of the result}

\subsection{Non-invariance of interior for endomorphisms of a cylinder}\label{sec:proof}
In this section we consider maps $f \in C^1(X)$, where $X$ is a cylinder $S^1 \times I$. The circle $S^1 = \mathbb{R}/\mathbb{Z}$ is viewed as horizontal and the segment $I$ as vertical. In Section~\ref{ss:emb} we will embed this example into an arbitrary surface.

We want our maps to have the following properties. 
\begin{enumerate}
    \item \label{cond:1} The map $f$ takes the cylinder $X$ into its interior and is homotopic to a skew product over a linear circle expanding map of large degree~$d$.
    It is partially hyperbolic with a forward-invariant horizontal cone field in which the tangent vectors are strongly expanded.
\end{enumerate}

\noindent We need to introduce several objects before we can state the rest of the properties.
Let $S^1_d = \mathbb R / d \mathbb Z$. Consider the $d$-fold covering $\pi$ from $\tilde{X} = S^1_d \times I$ to $X = S^1 \times I$ given by $(\theta,\; y) \mapsto {(\theta\,\mathrm{mod}\, 1,\; y)}$. As $f$ is homotopic to a skew product, by the covering homotopy theorem the map $f$ has a lift $F: X \mapsto \tilde{X}$, i.e., a map such that $\pi\circ F = f$. The lift $F$ is unique up to a shift by an integer. It will be more convenient to define the properties of $f$ in terms of its lift. Note that the lift is a map between two different spaces (the ``short cylinder'' and the ``long cylinder'').

Given a (not necessarily backward invariant) vertical cone field on $X$, we will call a curve \emph{almost vertical}, if it is tangent to this cone field. We can lift such a cone field to $\tilde{X}$ via $\pi$ and define almost vertical curves on~$\tilde{X}$. The lengths of the horizontal projections of almost vertical curves are bounded by some constant $\delta$ that depends on the apertures of the cones. We will consider only cone fields with constant small apertures and assume that $\delta < 1/2$.
Given a cylinder $C \subset X$, we will say that a curve \emph{cuts} this cylinder if it intersects the top and bottom boundaries of~$C$.
Finally, we consider closed segments $J_{in}, J_{out}, J_1, J_2 \subset I$ such that 
\begin{equation}\label{eq:inclusions}
    J_{in} \;\subset\; J_1 \cap J_2 \;\subset\; J_{out} \;\subset\; J_1 \cup J_2.
\end{equation}

We further assume that there is a vertical cone field with $\delta < 1/2$ such that the following holds.

\begin{enumerate}
    \setcounter{enumi}{1}
    \item\label{cond:2} For $j = 1, 2$, there is a closed arc $A_j \subset S^1$ and an integer $k_j$ such that 
    \[
        [k_j-\delta,\; k_j+1+\delta] \times J_j \;\subset\; 
        \mathrm{Int}\,(F(A_j \times J_{out}))
    \]
    and $F$ restricted to $A_j \times J_{out}$ is an orientation-preserving diffeomorphism onto the image with the inverse that preserves vertical cones and uniformly expands almost vertical curves.

    \item\label{cond:3} There is an integer $k_3$ such that
    \[
        F(S^1 \times I) \cap ([k_3-\delta,\; k_3+1+\delta] \times I) \;\subset\; [k_3-\delta,\; k_3+1+\delta] \times  \mathrm{Int}\,(J_{in})  
    \]
    and $F$ restricted to $F^{-1}([k_3-\delta,\; k_3+1+\delta] \times I)$ is an orientation-preserving diffeomorphism on its image such that its inverse preserves vertical cones and uniformly expands almost vertical curves.

    \item\label{cond:4} There is an open rectangle $\Pi = (B,D) \times \intt(J_{out}) \subset X$ such that the maximum of the vertical coordinate $y$ over $f(x),\; x \in X,$ is attained at some point $p \in \Pi$ while for any point $z$ outside $\Pi$ the $y$-coordinate of $f(z)$ is strictly less than this maximum.\footnote{Note that this condition is incompatible with $f$ being a local diffeomorphism. This is why we require the vertical cones to be backward invariant only for restrictions of $f$ to certain domains where it is a local diffeomorphism.}
\end{enumerate}

These conditions are $C^1$-open and are satisfied by a skew product that we describe in Section~\ref{sec:skew}. We denote by $\U_X \subset C^1(X)$ the open set of endomorphisms that satisfy these conditions. In order to show that this set is non-empty, we will construct in Section~\ref{sec:skew} a skew product map that satisfies the four conditions. Now we focus on proving for $U_X$ the claims analogous to those of Theorem~\ref{thm:main}.

\subsubsection*{Sketch of the proof}
Consider the stripe $S^1 \times J_{out}$. Conditions~\ref{cond:2}-\ref{cond:3} will be used in Lemma~\ref{lem:backward} to iterate an arbitrary small almost vertical curve in this stripe backwards using appropriate inverse branches of $f$ until we find in the preimage an almost vertical curve that cuts the whole cylinder~$X$. Any horizontal segment, when being iterated forward, will intersect this new curve, which yields that the stripe $S^1 \times J_{out}$ is contained in the generic limit set, and hence in the non-wandering set (Lemma~\ref{lem:stripe}). If the map is $C^2$-smooth, we can also use the distortion control argument and show (Lemma~\ref{lem:Milnor}) that this stripe is contained in the Milnor attractor. The same holds for generic $C^1$-maps (Lemma~\ref{lem:generic}). Informally speaking, condition~\ref{cond:4} says that $f$ folds a piece of the stripe $S^1 \times J_{out}$ in such a way that the image of the fold is at the boundary of $f(X)$. Hence, there are points in the interior of the attractor which are taken to the boundary of the whole image of $X$.

\subsubsection*{The chain of lemmas}
\begin{lem}\label{lem:backward}
Take a map $f \in \U_X$ and an almost vertical curve $L \subset S^1 \times J_{out}$. Then there is an almost vertical curve $S$ cutting $X$ and $n > 0$ such that $f^n(S) \subset L$.
\end{lem}

\begin{proof}
Set $L_0 = L$.
We construct a sequence $L_i$ of almost vertical curves lying in $S^1 \times J_{out}$ with $f(L_{i+1}) \subset L_i$ in the following way. We stop when we see that $L_i$ cuts the stripe $S^1 \times J_{in}$. Otherwise, since $J_{in} \;\subset\; J_1 \cap J_2$, the curve $L_i$ either lies in $S^1 \times J_1$ or in $S^1 \times J_2$. 

In the first case, when $L_i \;\subset\; S^1 \times J_1$, we lift\footnote{Note that we use the fact that our curve has horizontal projection of diameter at most $\delta$.} $L_i$ to $\tilde L_i \subset [k_1-\delta,\; k_1+1+\delta] \times J_1 \subset S^1_d \times I$. By condition~\ref{cond:2}, $\tilde L_i$ is inside the $F$-image of $A_1 \times J_{out}$.
Hence, we take $L_{i+1} = (F|_{A_1 \times J_{out}})^{-1}(\tilde L_i)$. This is again an almost vertical curve contained in $S^1 \times J_{out}$.
In the second case we proceed in the same way, but we lift $L_i$ to $\tilde L_i \subset [k_2-\delta, k_2+1+\delta] \times J_2$ and apply $(F|_{A_2 \times J_{out}})^{-1}$.

Since all $L_i$ are almost vertical and the chosen branches of $f^{-1}$ uniformly expand in the vertical cones, this process eventually stops when some $L_i$ cuts the cylinder $S^1 \times J_{in}$. Set $L' = L_i \cap (S^1 \times J_{in})$. Lift $L'$ to $\tilde L' \subset [k_3-\delta, k_3+1+\delta] \times J_{in}$.
By condition~\ref{cond:3}, the curve $S = F^{-1}(\tilde L')$ will cut~$X$.
\end{proof}

\begin{lem}\label{lem:stripe}
For any $f \in \U_X$ the cylinder $S^1 \times J_{out}$ is in $\omega(x)$ for a Baire-generic $x \in X$.
\end{lem}
\begin{proof}
Consider a small open ball $U \subset S^1 \times J_{out}$.
For any $N \in \mathbb N$, denote by $A_{N,U}(f)$ the set of points $x \in X$ such that for some $n>N$ we have $f^n(x) \in U$. This set is obviously open since $A_{N,U}(f) = \cup_{i=N+1}^{\infty} f^{-i}(U)$. We will show that it is dense in~$X$. Then the points in a residual set $\cap_N A_{N,U}(f)$ visit $U$ infinitely many times, i.e., for a generic point $x \in X$ the set $\omega(x)$ intersects the closure of $U$. Intersecting these residual sets over the balls $U$ with rational centers and radii, we conclude that $\omega(x)$ is dense in $S^1 \times J_{out}$ for generic $x \in X$. 

To show that any $A_{N,U}(f)$ is dense, we take an arbitrary almost vertical curve $L \subset U$ and construct an almost vertical curve~$S$ cutting $X$ with $f^k(S) \subset L$, as in Lemma~\ref{lem:backward}. 
Take any small (strictly) horizontal segment $H$ and iterate it forward. It will be expanded and eventually for all large enough $m$ the set $f^m(H)$ will be intersecting~$S$. If $m + k > N$, the points in $H \cap f^{-m}(S)$ belong to $A_{N,U}(f)$. Since $H$ was arbitrary, $A_{N,U}(f)$ is dense in~$X$.
\end{proof}

\begin{lem}\label{lem:Milnor}
If $f \in C^2(X) \cap \U_X$, the stripe $S^1 \times J_{out}$ is in $\omega(x)$ for almost all $x \in X$ w.r.t. the Lebesgue measure.
\end{lem}
\begin{proof}
We take a small open ball $U \subset S^1 \times J_{out}$ and show
that for any $N$ the set $A_{N,U}(f)$ has full measure. We fix a vertical curve $L \subset U$, construct a curve $S \subset f^{-\zeta}(L)$ that cuts $X$ as above, and consider a small neighborhood $V \supset S$ with $f^\zeta(V) \subset U$. For simplicity we assume that $V$ is the union of the images of $S$ under small horizontal shifts of size less than~$\e$. 

Let us show that, for every $N$, for Lebesgue almost any $x \in X$ there is $m\geq N$ with $f^m(x) \in V$, i.e., that the set $A_{N, V}(f)$ has full measure. Suppose that for some $N$ it does not. For brevity, denote this set $A_{N, V}(f)$ by $Y$ and its complement by $\overline{Y}$. By the Fubini theorem, the restriction of $\overline{Y}$ to some horizontal segment has a density point w.r.t. the one-dimensional Lebesgue measure. Then for any $k$ there is a subsegment $I_k$ such that
\[{\rm Leb}\,(Y\cap I_k)\; /\; {\rm Leb}\,(I_k) < 1/k.\] 

We fix some $k$ and iterate $I_k$ forward until $f^n(I_k)$ cuts\footnote{We say that a curve cuts through $V$ if there is an arc contained in $V$ with endpoints at the different components of $\partial V \setminus \partial X$, i.e. at $S - (\e, 0)$ and $S + (\e, 0)$.} through $V$ and $n$ is larger than the constant~$N$, i.e., we stop when both conditions are satisfied.
By the mean value theorem, there exist $x,y\in I_k$ such that
\[
{\rm Leb}\,(f^n(I_k\cap Y))=J^n(y)\, {\rm Leb}(Y\cap I_k) \text{ and }
{\rm Leb}\,(f^n(I_k)) = J^n(x)\,{\rm Leb}(I_k),
\]
where $J^n(x) = \|Df^n(x))|_{T_xI_k}\|$.

Moreover, we can assume that the length of $f^n(I_k)$ is bounded by some constant~$C$ independent of $n$ and $k$: when we see that the projection of $f^n(I_k)$ to $S^1$ is the whole circle and $n$ is greater than~$N$, we do one additional iteration to make sure that $f^n(I_k)$ cuts through $V$ and stop iterating. Since $f^n(I_k)$ is tangent to the horizontal cone field, it gives an upper bound on the length of $f^n(I_k)$. This allows to apply a distortion control argument. More precisely, since $f$ is $C^2$-smooth, we can apply Lemma 3.3 of~\cite{BV}\footnote{Lemma 3.3 of~\cite{BV} is stated for diffeomorphisms, but the proof is also valid in our case.}, which yields that there is a constant $K$ independent of $n$ (and $k$) such that $1/K<J^n(y)/J^n(x)<K$ for any $x, y \in I_k$. Therefore,
\[
    \frac{\mathrm{Leb}\,(f^n(I_k\cap Y))}{\mathrm{Leb}\,(f^n(I_k))} = 
    \frac{J^n(y)\, \mathrm{Leb}(I_k \cap Y)}{J^n(x)\, \mathrm{Leb}(I_k)} <
    \frac{K}{k},
\]

and the measure of $f^n(I_k \cap Y)$ is at most $CK/k$. By choosing a sufficiently large~$k$, we can make this measure arbitrarily small; this makes $n$ larger but does not affect $C$ and~$K$. On the other hand, it is not difficult to see that the length of a curve by which $f^n(I_k)$ cuts $V$ is bounded from below. Hence, some points of this intersection are in $f^n(\overline{Y})$, with $n > N$, which contradicts the definition of $\overline{Y}$. The contradiction shows that the set $Y = A_{N, V}(f)$ has full Lebesgue measure for every~$N$. 
As $A_{N, V} \subset A_{N, U}$, the set $A_{N, U}$ also has full measure. The intersection of these sets over $N \in \mathbb{N}$ and over all balls $U \subset S^1 \times J_{out}$ with rational centers and radii is the required set of full measure.  
\end{proof}

\begin{lem}\label{lem:generic}
For a $C^1$-generic endomorphism $f$ in $\U_X$, the stripe $S^1 \times J_{out}$ is in $\omega(x)$ for almost all $x \in X$ w.r.t. the Lebesgue measure.
\end{lem}

\begin{proof}
Take a small open ball  $U \subset S^1 \times J_{out}$. As above, $A_{N,U}(f)$ is the set of points $x$ such that for some $n>N$ we have $f^n(x)\in U$. Obviously, $A_{N,U}(f) = \cup_{i=N+1}^{\infty} f^{-i}(U)$.
The measure of $f^{-1}(U)$ lower semi-continuously depends on $f$, so the measure of the union ${\rm Leb}(\cup_{i=N+1}^{K} f^{-i}(U))$ also is a lower semi-continuous function of $f$, for any $K$. Therefore ${\rm Leb}(A_{N,U}(f)) = \sup_{K} {\rm Leb}(\cup_{i=N+1}^{K} f^{-i}(U))$ is a lower semi-continuous function.\footnote{We can choose $K$ such that the measure of the finite union $\cup_{i=N+1}^{K} f^{-i}(U)$ is $\e$-close to $\mathrm{Leb}(A_{N,U}(f))$. When $f$ is slightly perturbed, the measure of the finite union decreases by at most $\e$, so $\mathrm{Leb}(A_{N,U}(f))$ can decrease buy at most $2\e$, hence lower semicontinuity.} Thus the set $B(N,U,m)$ of all $C^1$-maps $f$ such that ${\rm Leb} (A_{N,U}(f)) > 1-1/m$ is open. The set $B(N,U,m)$ is also dense in $\U$, as it contains any $C^2$-map (by Lemma~\ref{lem:Milnor}) and $C^2(X)$ is dense in $C^1(X)$.

For $f$ in a residual set $B(N,U)=\cap_m B(N, U, m)$ we have ${\rm Leb} A_{N,U}(f)=1$.
Finally, we intersect the sets $B(N,U)$ over $N \in \mathbb{N}$ and balls $U \subset S^1 \times J_{out}$ with rational centers and radii and obtain a residual subset $\R_X$ of $\U_X$ with the required property.

\end{proof}

\begin{lem}[Example on the cylinder] \label{l:main}
\;

\begin{itemize}
    \item For any cylinder endomorphism in $\mathcal U_X$, the non-wandering set and the generic limit set both have non-invariant interior.
    \item For any cylinder endomorphism in the set $\mathcal R_X$ from Lemma~\ref{lem:generic}, the Milnor attractor also has non-invariant interior.
    \item $\mathcal R_X$ is a $C^1$-residual subset of $\mathcal U_X$; it contains all $C^2$-endomorphisms of $\U_X$.
\end{itemize}
\end{lem}

\begin{proof}
Let $f \in \U_X$ and $A$ be either the non-wandering set of the generic limit set. Consider the point $p$ from condition~\ref{cond:4}. The point $p$ is in the open cylinder $S^1 \times J_{out}$ contained in $A$ (Lemma~\ref{lem:stripe}), so $p \in \Int \; A$. By invariance, $f(p)\in A$. Since $A$ must be contained in $f(X)$, any point of $A$ at the boundary of $f(X)$ is at the boundary of~$A$.
So, $f(p) \in \partial A$. We have found a point in the interior of $A$ that is taken by $f$ to the boundary of~$A$. This proves the first claim.

Now take any $f \in \R_X$. For such $f$, the Milnor attractor $A_{Mil}$ contains the stripe $S^1 \times J_{out}$ by Lemma~\ref{lem:generic}, and the previous argument shows that the interior of $A_{Mil}$ is not invariant as well. This proves the second claim.
Finally, $\R_X$ was defined in Lemma~\ref{lem:generic} as a countable intersection of residual sets that contain every $C^2$-map in $\U_X$, which yields the third claim.
\end{proof}

\subsection{The initial skew product}\label{sec:skew}
In this section we construct a skew product map that satisfies conditions~\ref{cond:1}-\ref{cond:4} of Section~\ref{sec:proof}.
We set $I = [-1, 1]$ and define our skew product $f \colon X \to X,\; X = S^1\times [-1,1],$ by the formula
\[
(\t,y) \mapsto (d\cdot \t, \ef_\t(y)),
\]
where $\t\in S^1,\; y\in [-1,1]$, and the degree $d\in \mathbb N$ and the fiber maps $f_\t$ are to be specified.

First we need to choose the subsegments $J_{in}, J_{out}, J_1, J_2$ in $I$ that satisfy~\eqref{eq:inclusions}. One can fix some numbers $a, b$ such that  $1 > 3a > b > 2a > 0$ and take $J_1=[-a, b]$, $J_2=[-b,a]$, $J_{out}=[-2a,2a]$, $J_{in}=[-a/2,a/2]$.

We introduce two contractions $\psi_{1}$ and $\psi_{2}$ that take the segment $I = [-1,1]$ strictly in itself so that the interior of the image of $J_{out}$ contains $J_{1}$ and $J_2$, respectively, \footnote{One can choose segments $V_i, \; i = 1, 2,$ such that $J_i \subset \intt{V_i}$ and the length of $V_i$ is less than the length of $J_{out}$ for $i=1,2$. Now set $\psi_{i}$ to be the linear contraction of $J_{out}$ onto $V_i$. If $V_i$ are sufficiently close to $J_i$, $\psi_{i}$ maps $J_{out}$ strictly into $I$, and can be continued to a non-linear contraction of $I$ into itself.}
and a map $\psi_3$ that linearly contracts $I$ to a segment in the interior of~$J_{in}$.
We also need a special map $\psi_4 \colon I \to I$ of the form $y \mapsto \alpha y^2+ \beta$, where $\beta \in (\max_{y\in I}\{\psi_1(y), \psi_2(y)\}, 1)$ and $\alpha$ is negative and small in absolute value,  so that we have $\psi_4(I)\subset I$. For this map, the maximum is attained at $y = 0$ and is equal to $\beta$.

Fix small positive $\e$ and choose four disjoint circle arcs $A_1, A_2, A_3, A_4$ with distance between them greater than~$\e$. Now we can describe our fiber maps $f_\t$: we set $\ef_\t=\psi_i$ on $A_i$, and on $S^1\setminus (\cup A_i)$ the maps $\ef_\t$ can be pretty much arbitrary, provided that they smoothly depend on $\t \in S^1$ and satisfy $\ef_\t < \beta$ for $\t$ outside~$A_4$. If we choose $d$ large enough, our skew product will be partially hyperbolic with any prescribed invariant horizontal cone field. Moreover, we will have 
$d \cdot A_i \supset [k_i - \delta, k_i+1 + \delta]$ for some integers $k_i$ and $i = 1,\dots, 4$.

Now it is easy to see that if one takes a family of sufficiently narrow vertical cones (to define almost vertical curves) the skew product $f$ will satisfy conditions~\ref{cond:1}-\ref{cond:3} of Section~\ref{sec:proof}. Condition~\ref{cond:4} is satisfied with $(B, D)$ being a small neighborhood of the arc $A_4$: the maximum $\beta$ of the vertical coordinate of $f(x)$ is attained for $x \in A_4 \times \{0\}$, and the maximum over $X \setminus ((B,D) \times \intt(J_{out}))$ is less than~$\beta$.

\subsection{Extending the skew product to an arbitrary surface} \label{ss:emb}

Let us now explain why the skew product on the cylinder we have built in Section~\ref{sec:skew} can be extended as a $C^\infty$-endomorphism $\hat f$ of an arbitrary two-dimensional surface~$M$.

Denote the skew product built in Section~\ref{sec:skew} by $\bar f$.
Consider any embedding of the cylinder~$X$ into~$M$. Consider a small neighborhood $U$ of $X$ that is also a topological cylinder. Our skew product $\bar f$ can be continued\footnote{We can take an advantage of the skew product structure and construct a skew product defined on $S^1 \times [-1.1, 1.1]$ such that its restriction to $X=S^1 \times [-1, 1]$ coincides with $\bar f$ in the same way we have built~$\bar f$.}  as a map $\bar f: U \mapsto U$. Let $\varphi: M \mapsto [0, 1]$ be a $C^\infty$-function with $\varphi=1$ on $X$ and $\varphi=0$ on $M \setminus U$.
Then $\hat f = \varphi \bar f + (1-\varphi) \id$ (here $\id$ is the identity map) is a $C^{\infty}$-endomorphism on $M$ such that its restriction to $X$ coincides with~$\bar f$.

\subsection{Proof of the main theorem}

 Consider an arbitrary two-dimensional surface $M$. In Sections~\ref{sec:skew} and \ref{ss:emb} we have embedded the cylinder $X$ in $M$ and constructed an endomorphism $\hat{f}$ such that its restriction to $X$ is in~$\mathcal U_X$. Take as $\mathcal U$ the set of all $C^1$-endomorphisms $f$ of $M$ such that $f|_X$ is in~$\mathcal U_X$. This set is open (as $\mathcal U_X$ is open) and non-empty, since it contains~$\hat f$. 
 
 We denote the non-wandering set of an arbitrary map $g$ by $\Omega(g)$. Take any $f \in \mathcal U$.
 By Lemma~\ref{l:main}, the interior of $\Omega(f|_X)$ is non-invariant. 
 Note that $X$ is a trapping region for $f$. Thus, $\Omega(f) \cap X = \Omega(f|_X)$, and $f(\Omega(f|_X)) \subset \intt X$. This implies that the interior of $\Omega(f)$ is non-invariant as well: any point in $\Omega(f|_X)$ mapped to the boundary of this set is mapped to the boundary of $\Omega(f)$ as well.
 The same reasoning applies to the generic limit set. 

 Take as $\mathcal R$ the set of all $C^1$-endomorphisms $f$ of $M$ such that $f|_X$ is in~$\mathcal R_X$. This set is residual in $\mathcal U$ and contains all $C^2$ endomorphisms; this is inherited from the properties of $\mathcal R_X$ (Lemma~\ref{l:main}).
 By Lemma~\ref{l:main} the Milnor attractor of $f|_X$ has non-invariant interior for all $f \in \mathcal R$. Applying the same reasoning as for the non-wandering set, we see that the Milnor attractor of $f$ has non-invariant interior as well.

\appendix

\section{A trivial open set of continuous maps whose attractor has non-invariant interior} \label{s:non-manifold}

An example of attractor with non-invariant interior can be easily constructed if the phase space does not have to be connected and does not have to be a manifold.
Indeed, let our phase space be $[-1,1]\cup 2$, and let $f([-1,1]) = \{2\}$ and $f(2) = 0$. Clearly, $f$ is continuous. Every orbit of $f$ contains the point~$2$, and therefore contains a ball of radius $\frac{1}{2}$ centered at~$2$ (this ball coincides with $\{2\}$). The attractor $A$ of $f$ is $\{0, 2\}$, with 2 being an interior point of~$A$ and 0 being a boundary point. Moreover, the interior point 2 of $A$ is taken into a boundary point~0. Therefore, the interior is non-invariant.

 It is easy to see that small perturbations of this map do not change its behavior, and the interior part of the attractor remains non-invariant. Therefore, we have an open set with the required property.

\section{No attractors with non-invariant interior in \texorpdfstring{$C^1$}{C1}-generic one-dimensional dynamics.} \label{s:1D}

In this section we show that $C^1$-generic maps of a circle or a segment cannot have attractors with non-invariant interior. This follows from a theorem of M. Jacobson~\cite[Theorem~A]{J}. Hence, our two-dimensional example from Section~\ref{sec:proof} has the smallest possible dimension when it comes to the $C^1$-results on interior non-invariance. 

First we consider generic circle maps. Denote by $\Sigma$ the set of all points of $S^1$ which do not belong to a basin of a hyperbolic attracting periodic orbit. This set $\Sigma$ is closed.
According to M. Jacobson~\cite[Theorem~A]{J}, there is an open and dense subset of $C^1(S^1)$ in which every map either is conjugated to a circle covering (and in this case $\Sigma=S^1$) or has totally disconnected set~$\Sigma$.  

In the first case the generic limit set is the circle. For any other closed forward invariant attractor without wandering points (in particular, for $A_{Mil}$) we can use the following idea. Suppose this attractor has non-empty interior. Then it contains an interval, and, by invariance, it contains every periodic orbit that intersects this interval. But for maps conjugated to circle coverings every pair of open sets is connected by a periodic orbit. Therefore, the attractor is dense in $S^1$ and, since it is closed, coincides with the whole~$S^1$.

In the second case we argue that the basins of sinks do not belong to the attractor, except for the sinks themselves. Therefore the attractor is contained in the union of~$\Sigma$ and the set of hyperbolic sinks. These sets are separated by the basins of attraction, and both are totally disconnected. Thus our attractor is totally disconnected and does not contain an interval.

Now we consider the case of maps of a segment into itself. Any map of a segment, even orientation reversing, can be viewed as a restriction of a circle map: we just add a repelling domain with a single hyperbolic source. Hence, a locally generic map of a segment may be viewed as a restriction of a locally generic circle map with a totally disconnected set $\Sigma$. Therefore, its attractor has empty interior.\\

\section{Generalization to manifolds of higher dimension}
\label{s:3d}

In this section we explain how to modify the example presented above so that we had the same result as in the main Theorem~\ref{thm:main}, but with $M$ being an arbitrary manifold of dimension at least 2, and not just an arbitrary surface.

We fix the dimension of the phase space $M$, say $n + 1$, construct an open set $\U_X$ of endomorphisms of $X = S^1 \times D^n$, where $D^n$ is a unit $n$-dimensional disc, and then ``extend'' the example to $M$ exactly as in Section~\ref{ss:emb}.
The argument for $\U_X$ must work as outlined in the ``Sketch of the proof'' part of Section~\ref{sec:proof}, but with almost vertical curves replaced with almost vertical hypersurfaces (we will write ``vertical surfaces'' for brevity). We take an $n+1$ dimensional analogue $S^1 \times B$ of $S^1 \times J_{out}$, fix an arbitrary vertical surface inside it and, informally, iterate it backwards until we find, in the preimage, a vertical surface that cuts $X$. It is this backward iteration part that must be explained, i.e., we must discuss the conditions that will give us an analogue of Lemma~\ref{lem:backward}. After that, Lemmas~\ref{lem:stripe}-\ref{l:main} and their proofs can be repeated almost verbatim, which would finish the proof.

\subsubsection*{The modification}
There is an example of an $n$-dimensional box that is covered by the union of its images under two affine contractions.\footnote{See, e.g., \cite[Section 3.1]{V}: one can fix $\lambda < 1$ close to $1$, take a box $B = [-r_1, r_1] \times \dots \times [-r_n, \; r_n]$ with $r_{j+1} < \lambda r_j\; (\forall j)$ and $\lambda r_n > r_1/2$, and apply first the isometry $R \colon (x_1, \dots, x_n) \mapsto ((-1)^{n+1}x_n, x_1, \dots, x_{n-1})$ and then the contraction $\Lambda = \lambda\cdot \id$. Then $B$ is covered by the images of $\Lambda R(B)$ under two appropriate translations along the first coordinate.} Let $B \subset D^n$ be a (closed) box like that and $\psi_1, \psi_2$ be the contractions: we have $B \subset \intt(\psi_1(B)) \cup \intt(\psi_2(B))$. We may assume, after changing the size of $B$, that $\psi_1, \psi_2$ take $D^n$ into its interior. Let $J_1 \subset \psi_1(B),\; J_2\subset \psi_2(B)$ be two slightly smaller boxes such that the union of their interiors covers $B$. We replace condition~\ref{cond:2} from Section~\ref{sec:proof} with the following.

\medskip
\noindent{\bf Condition $\hat{2}$.} For $j = 1, 2$, there is a closed arc $A_j \subset S^1$ and an integer $k_j$ such that 
    \[
        [k_j-\delta,\; k_j+1+\delta] \times J_j \;\subset\; 
        \mathrm{Int}\,(F(A_j \times B))
    \]
    and $F$ restricted to $A_j \times B$ is an orientation-preserving diffeomorphism onto the image with the inverse that preserves and uniformly expands vertical cones.
    
\medskip
Here $F\colon S^1 \times D^n \to S^1_d \times D^n$ is the lift of $f\colon X \to X$, and some vertical cone field with small aperture is fixed. This condition is open, and it is easy to see that it is satisfied for the initial skew product if we take the two affine contractions $\psi_1, \psi_2$ as the corresponding fiber maps of the skew product.

Note that since $B \subset \intt(J_1) \cup \intt(J_2)$, there is $\e > 0$ such that the box $B$ is a union of two subsets $B_1$ and $B_2$ whose $\e$-neighborhoods are contained in, respectively, $J_1, J_2$. One can take any $\e$ smaller than the Lebesgue number of the cover $\intt(J_1) \cup \intt(J_2)$; we also make sure that $\e < \delta$. Then the $\e$-neighborhood of any point in $S^1 \times B$ can be appropriately lifted to $S^1_d \times D^n$ (as in Lemma~\ref{lem:backward}) and then iterated backwards by $F|_{A_1 \times B}$ or $F|_{A_2 \times B}$, with preimage being in $S^1 \times B$.

After replacing $\e$ with a smaller number, we may assume that any vertical surface in $S^1 \times B$ of diameter less than $\e$ is expanded when we take these preimages. Indeed, by condition~{$\hat{2}$} the vertical cones are uniformly expanded when we iterate backwards. This implies that small vertical surfaces are expanded, e.g., in the following sense: if a piece of a vertical surface projects into an $r$-ball in $D^n$, then the preimage of this piece must contain a $\sigma r$-ball in its projection, where $\sigma > 1$ is a universal constant that governs the expansion. It is easy to see now that, starting with an arbitrary vertical surface in $S^1 \times B$, we can iterate it backwards until we find a vertical surface in the preimage with a projection to $D^n$ that contains a ball of diameter $\e/2$. We will call such vertical surfaces $\e/2$-large.

Let $\{p_1, \dots, p_N\}$ be an $\e/8$-net in $B$ and let $\psi_j, \; j = 3,\dots, N+2$ be affine contractions of $D^n$ into the $\e/8$-neighborhoods of the corresponding points $p_j$. If we utilize these maps as fiber maps of the initial skew product $f$ (for $\t$ in arcs $A_j, \; j = 3, \dots, N+2$, instead of the old $\psi_3$ that was used for $\t \in A_3$), we will have that for any $\e/2$-large vertical surface the $f$-preimage contains a vertical surface that cuts $X$. The same will hold for any map close to this skew product.
This yields the analogue of Lemma~\ref{lem:backward}.

Finally, it is easy to construct a map $\psi_{N+3}$ that folds $D^n$ so that the fold intersects $\intt(B)$ and the image of the fold is close to the boundary of $D^n$ to make sure that the analogue of condition~\ref{cond:4} holds.

\vskip 5mm

\noindent Stanislav Minkov

\noindent {\footnotesize{E-mail : stanislav.minkov@yandex.ru}}

\vskip 5mm

\noindent Alexey Okunev

\noindent {\footnotesize{E-mail : abo5297@psu.edu}}

\vskip 5mm

\noindent Ivan Shilin

\noindent {\footnotesize{E-mail : i.s.shilin@yandex.ru}}

\end{document}